\documentclass[11pt,reqno]{amsart}

\numberwithin{equation}{section}

\usepackage{color,graphics,epsfig,a4wide,bm,mathrsfs}

\DeclareMathAlphabet{\mathpzc}{OT1}{pzc}{m}{it}

\newcommand{\calD}{\mathcal{D}}

\newcommand{\calF}{\mathcal{F}}

\newcommand{\calL}{\mathcal{L}}

\newcommand{\calO}{\mathcal{O}}

\newcommand{\calS}{\mathcal{S}}

\newcommand{\cals}{\mathpzc{s}}

\newcommand{\mC}{\mathbb{C}}

\newcommand{\mN}{\mathbb{N}}

\newcommand{\mR}{\mathbb{R}}

\newcommand{\mZ}{\mathbb{Z}}

\newcommand{\bxi}{{\bm{\xi}}}
\newcommand{\balpha}{{\bm{\alpha}}}
\newcommand{\bpartial}{{\bm{\partial}}}

\newcommand{\nm}{\,\rule[-.6ex]{.13em}{2.3ex}\,}

\newtheorem{theorem}{Theorem}[section]
\newtheorem{lemma}[theorem]{Lemma}
\newtheorem{corollary}[theorem]{Corollary}

\theoremstyle{definition}

\newtheorem{example}[theorem]{Example}

\theoremstyle{definition}

\theoremstyle{definition}

\begin{document}

\keywords{partial differential equations,
   distributions that are periodic in the spatial directions, Fourier transformation,
  spatially invariant systems}


\title[Division problem for distributions]{Division problem for spatially periodic distributions }

\date{May 30, 2013}

 \author{Amol Sasane}
 \address{Department of Mathematics, London School of Economics,
      Houghton Street, London WC2A 2AE, United Kingdom.}
 \email{sasane@lse.ac.uk}

 \author{Peter Wagner}
 \address{Faculty of Technical Sciences, Universit\"at Innsbruck, Technikerstr.~13,
 A-6020 Innsbruck, Austria.}
 \email{wagner@mat1.uibk.ac.at}

\begin{abstract}
 We give a sufficient condition for the surjectivity of partial differential operators with constant
 coefficients  on the space of distributions on $\mR^{n+1}$ (here we think of there being $n$
 space directions and one time direction)
 that are periodic in the spatial directions and tempered in the time direction.
\end{abstract}

\maketitle

\section{Introduction}

An important milestone in the general theory of partial differential equations is
the solution to the division problem: Let $D$ be a nonzero partial differential operator with constant
coefficients and $T$ be a distribution; can one find a distribution $S$ such that
$DS=T$? That this is always possible was established by Ehrenpreis \cite{Ehrc}. See also
\cite{Ehr}, \cite{Ehrb}, \cite{Ehrd}, \cite{Ehre}, \cite{Hor}, \cite{Loj},
\cite{Wag} for the solution of  various avatars of the division problem
for different spaces of distributions.

In this article, we study the division problem in spaces of distributions on $\mR^{n+1}$
(where we think of there being $n$ space directions and one time direction)
 that are periodic in the spatial directions. The study of such solution spaces arises
 naturally in control theory when one considers the so-called ``spatially invariant systems''; see
 \cite{Sas}. In the ``behavioural  approach'' to control theory for such
 spatially invariant systems,  a  fundamental question is whether this class of distributions is
 an injective module over the ring of partial differential operators with constant coefficients; see
 for example \cite{Woo}. In light of this, one can first  ask what happens with the division
 problem. Thus besides being a purely mathematical question that fits in the
 classical theme mentioned in the previous paragraph, there is also a
 behavioural control theoretic motivation for studying the division problem for
 distributions that are periodic in the spatial directions.
 Upon taking Fourier transform with respect to the spatial variables,
 the problem amounts to the following.

\medskip

\noindent {\bf Problem:} For which $P(\tau,\bxi)\in \mC[\tau,\xi_1,\cdots,\xi_n]$ is
 $
\displaystyle P\Big(\frac{d}{dt},i\bxi\Big):X\rightarrow X
$
surjective, where
$$
X=\{(T_\bxi)=(T_\bxi)_{\bxi \in \mZ^n}\in \calD'(\mR)^{\mZ^n}:\;\forall \varphi \in \calD(\mR), \;\exists
k\in \mN: \;\forall \bxi\in \mZ^n, \;|\langle \varphi, T_\bxi\rangle |\leq
k\cdot(1+\nm \bxi \nm)^k\}.
$$
($\nm \cdot\nm$ denotes the $1$-norm in $\mR^n$.)

\medskip

\noindent An obvious necessary condition is that
$$
\textrm{ for all }\bxi \in \mZ^n, \;P\Big(\frac{d}{dt},i\bxi\Big)\not\equiv 0.
$$
However, that this condition is not sufficient is demonstrated by considering the following example.

\begin{example}
\label{example_1}
 Let $c=\displaystyle \sum_{j=1}^\infty \frac{1}{2^{j!}}
$ (a ``Liouville number'').  With $p_k:=\displaystyle \sum_{j=1}^k 2^{k!-j!}$, $q_k:=2^{k!}$,
$$
\left|c- \frac{p_k}{q_k}\right|
=
\sum_{j=k+1}^\infty \frac{1}{2^{j!}}
=\frac{1}{2^{(k+1)!}} + \frac{1}{2^{(k+2)!}}+\cdots<
2\cdot \frac{1}{2^{(k+1)!}}
\leq \left(\frac{1}{2^{k!}}\right)^k
=\frac{1}{q_k^k},\;\;
k\in \mN.
$$
Now take $P=\xi_1+c\xi_2$. 
Consider the equation $PS=T$ where $T_\xi=1$ for $\xi\ne(0,0)$ and $T_{(0,0)}=0$
and hence $S_\xi=1/(\xi_1+c\xi_2)$ for $\xi\ne(0,0)$ and $S_{(0,0)}$ is arbitrary.
But $S$ does not belong to $X$ because otherwise there would exist an $m$ such that
$$
\frac{1}{|\xi_1+c\xi_2|}\leq m(1+|\xi_1|+|\xi_2|)^m \textrm{ for all }\xi_1,\xi_2\in \mZ\setminus\{0\},
$$
and in particular,  with $\xi_1=-p_k $, and $\xi_2=q_k$, $k\in \mN$,
$$
q_k^{k-1}\leq \frac{1}{|\xi_1+c\xi_2|}\leq m (1+p_k+q_k)^m\leq m (q_k+q_k+q_k)^m = m3^m q_k^m,
$$
a contradiction. 
\hfill$\Diamond$
\end{example}

\noindent We consider a simpler situation and set
$$
Y=\{(T_\bxi)\in \calS'(\mR)^{\mZ^n}:\;\forall \varphi \in \calS(\mR), \;\exists
k\in \mN: \;\forall \bxi\in \mZ^n, \;|\langle \varphi, T_\bxi\rangle |\leq
k\cdot(1+\nm \bxi \nm)^k\}.
$$
Our main result is the following:

\begin{theorem}
 Let $P(\tau,i\bxi)\in \mC[\tau,\bxi]=\mC[\bxi][\tau]$, and for $\bxi\in \mZ^n$,
 $$
 P(\tau,i\bxi)=c_{\bxi} \cdot \prod_{j=1}^{m_\bxi} (\tau-\lambda_{j,\bxi}),
 $$
 with $m_\bxi \in \mN_0$, $c_\bxi \in \mC\setminus \{0\}$, $\lambda_{1,\bxi},\cdots, \lambda_{m_\bxi, \bxi}\in \mC$.
 $($The roots
 $\lambda_{j,\bxi}$ are arbitrarily arranged.$)$
 Let
 $$
 d_\bxi :=\left\{\begin{array}{cl}
                  1 & \textrm{if for all }j=1,\cdots,m_\bxi, \;\textrm{\em Re}(\lambda_{j,\bxi})
                  =0,\phantom{\displaystyle \inf_{j}}\\
                  \displaystyle \min_{j:\textrm{\em Re}(\lambda_{j,\bxi})\neq 0}
                  |\textrm{\em Re}(\lambda_{j,\bxi})|& \textrm{otherwise}.
                 \end{array}\right.
 $$
 If
 \begin{enumerate}
  \item $(c_\bxi^{-1}) \in \cals'(\mZ^n)$ and
  \item $(d_\bxi^{-1}) \in \cals'(\mZ^n)$,
 \end{enumerate}
then $P\left(\displaystyle\frac{d}{dt}, i\bxi\right):Y\rightarrow Y$ is surjective.
\end{theorem}

From Example~\ref{example_1}, it follows that the first condition is not superfluous.
Here is an example demonstrating  that the second condition is not superfluous either.

\begin{example}
 Take
 $$
 P\left(\displaystyle\frac{d}{dt}, i\bxi\right):=\frac{d}{dt} +\xi_1+c \xi_2
 $$
 with the same $c$ as in Example~\ref{example_1}, and
  $
 T_\bxi:=1$, $\bxi \in \mZ^2.
 $ Then the constant $S_\bxi=1/(\xi_1+c\xi_2)$ is the only solution in $\calS'(\mR)$ of the equation
$$
P\left(\displaystyle\frac{d}{dt}, i\bxi\right)S_\bxi=T_\bxi=1,
$$ 
for $\xi\ne (0,0)$, but the sequence
  $$ 
  (S_\bxi) =\left(\displaystyle\frac{1}{\xi_1+c\xi_2}\right)
  $$ 
  does not belong to $Y$ and
hence $(1)_{\bxi \in \mZ^2}$  is not contained in range of $P.$
  \hfill$\Diamond$
\end{example}

\section{Preliminaries}

There holds that
$$
\begin{array}{ccccc}
Y&\simeq &\calL(\calS(\mR),\cals'(\mZ^n))&\simeq &  \calS'(\mR)\widehat{\otimes} \cals'(\mZ^n).\\
(T_\bxi)&\longmapsto & (\varphi\mapsto \langle \varphi, T_\bxi\rangle)
\end{array}
$$
(For the first isomorphism we use the Closed Graph Theorem,
while the second isomorphism follows from \cite[Proposition~50.4]{Tre}.)
Also, by \cite[Theorem 51.3, p.~528, and Corollary to Theorem 51.6, p.~531]{Tre},
$$
\calS'(\mR)\widehat{\otimes} \cals'(\mZ^n)\simeq \calS'(\mR)\widehat{\otimes} \calS'((\mR/\mZ)^n)
\simeq \calS'(\mR\times (\mR/\mZ)^n)=\widehat{Y},
$$
that is, $\widehat{Y}$ is the dual of a Fr\'echet space and there holds that
$$
\begin{array}{cc}
\textrm{for all }\widehat{T}\in \widehat{Y} \textrm{ there exists } k\in \mN \textrm{ such that for all }
\psi \in \calS(\mR \times (\mR/\mZ)^n),\\
|\langle \psi , \widehat{T} \rangle |
\leq k\cdot \displaystyle \sum_{\nm \balpha\nm \leq k} \|(1+|t|)^k \bpartial^\balpha\psi\|_\infty .\end{array}
$$
Hence it follows that in $Y$:
$$
\begin{array}{cc}
\textrm{for all }(T_\bxi) \in Y \textrm{ there exists } k\in \mN \textrm{ such that for all }
\varphi \in \calS(\mR) \textrm{ and all } \bxi \in \mZ^n,\\
|\langle \varphi , T_\bxi \rangle |
\leq k\cdot (1+\nm \bxi\nm )^k \cdot  \displaystyle \sum_{j=0}^k
\|(1+|t|)^k \varphi^{(j)} \|_\infty .\end{array}
$$
In particular for $\varphi \in \calS(\mR)$ and $t\in \mR$,
\begin{eqnarray*}
|(\varphi \ast \check{T_\bxi})(t)|
&=&
|\langle (\tau\mapsto\varphi(t+\tau)),T_\bxi \rangle |
\leq  k\cdot (1+\nm \bxi \nm )^k \cdot
\sum_{j=0}^k \left\|\tau\mapsto \Big((1+|t-\tau|)^k  \varphi^{(j)}(\tau)\Big) \right\|_\infty
\\
&\leq & C_{\varphi,k} \cdot (1+\nm \bxi\nm)^k \cdot (1+|t|)^k.
\end{eqnarray*}

We will also  need the following lemma.

\begin{lemma}
 \label{lemma_poly_growth}
 Consider a monic polynomial
 $$
 P=\tau^d +c_{d-1,\bxi} \tau^{d-1} +\cdots+ c_{1,\bxi}\tau +c_{0,\bxi} \in \cals'(\mZ^n)[\tau].
 $$
 For $\bxi\in \mZ^n$, we factorize
 $$
 P(\tau,i\bxi)= \prod_{j=1}^{d} (\tau-\lambda_{j,\bxi}),
 $$
 with $\lambda_{1,\bxi},\cdots, \lambda_{d,\bxi}\in \mC$.  $($The roots
 $\lambda_{j,\bxi}$ are arbitrarily arranged.$)$
  Then  $( \lambda_{j,\bxi})_\bxi\in \cals'(\mZ^n)$, $j=1,\cdots, d$.
\end{lemma}
\begin{proof}
 Let $(\tau +\alpha_\bxi)(\tau^{d-1}+ b_{d-2,\bxi}\tau^{d-2}+\cdots+b_{1,\bxi}\tau +b_{0,\bxi})=P$.
 Then by comparing coefficients of the powers of $\tau$, we obtain
 \begin{eqnarray*}
  \alpha_\bxi \cdot b_{0,\bxi}&=&c_{0,\bxi},\\
  \alpha_\bxi \cdot b_{1,\bxi}+b_{0,\bxi} &=&c_{1,\bxi},\\
  &\vdots & \\
   \alpha_\bxi \cdot b_{d-2,\bxi}+b_{d-3,\bxi} &=&c_{d-2,\bxi},\\
   \alpha_\bxi + b_{d-2,\bxi} &=&c_{d-1,\bxi}.
 \end{eqnarray*}
 We first show that $(\alpha_\bxi)\in \cals'(\mZ^n)$. Suppose this is not true.
 Then there exists a sequence $(\bxi_k)_{k}$ such that
 $$
 \lim_{k\rightarrow\infty}\nm \bxi_k \nm=\infty
 $$
 and $|\alpha_{\bxi_k}|>k (1+\nm \bxi_k \nm)^k$ for all $k$. But then from the first equation in the above equation array,
 it follows that
 $$
 |b_{0,\bxi_k}|= \frac{|c_{0,\bxi_k}|}{|\alpha_{\bxi_k}|}\stackrel{k\rightarrow \infty}{\longrightarrow} 0.
 $$
 Then from the second equation in the above equation array, we also obtain that
 $$
|b_{1,\bxi_k}|= \frac{|c_{1,\bxi_k}-b_{0,\bxi_k} |}{|\alpha_{\bxi_k}|}\stackrel{k\rightarrow \infty}{\longrightarrow} 0.
 $$
 Proceeding in this manner, we get eventually that
 $$
 1\leq \frac{|c_{d-1,\bxi_k}-b_{d-2,\bxi_k} |}{|\alpha_{\bxi_k}|}\stackrel{k\rightarrow \infty}{\longrightarrow} 0,
 $$
 a contradiction.

 We will be done once we show that $(b_{d-2,\bxi}),\cdots,  (b_{1,\bxi}) , (b_{0,\bxi})$ all  belong
 to $\cals'(\mZ^n)$ by an inductive argument. Suppose that $k$ is the least index such that
  $(b_{k,\bxi})\not\in \cals'(\mZ^n)$. But then by a similar argument as above, it follows
  from
  $$
    \alpha_\bxi \cdot b_{k+1,\bxi}+b_{k,\bxi} =c_{k+1,\bxi}
    $$
    that
    $(b_{k+1,\bxi})\not\in \cals'(\mZ^n)$ (since we have
    already established that $(\alpha_\bxi)\in \cals'(\mZ^n)$). Proceeding in this manner,
    we eventually obtain that  $b_{d-2,\bxi}\not\in \cals'(\mZ^n)$,
    which clearly contradicts the last equation in the above equation array, namely that
    $\alpha_\bxi + b_{d-2,\bxi} =c_{d-1,\bxi}$.
\end{proof}

\section{Proof of the main result}

\begin{proof} (a) The pointwise multiplication map
$$
c_\bxi: \cals'(\mZ^n)\rightarrow  \cals'(\mZ^n): (\alpha_\bxi) \mapsto (c_\bxi \cdot \alpha_\bxi)
$$
is a topological vector space isomorphism, thanks to the first assumption that $(c_\bxi^{-1}) \in \cals'(\mZ^n)$. So it
is enough to prove the surjectivity of
$$
Q:=\frac{1}{c_\bxi} P\left(\displaystyle\frac{d}{dt}, i\bxi\right):Y\rightarrow Y,
$$
that is, of $Q$ given by
$$
Q(T_\bxi)=\prod_{j=1}^{m_\bxi}\left(\displaystyle\frac{d}{dt}-\lambda_{j,\bxi}\right) T_\bxi.
$$
By Lemma~\ref{lemma_poly_growth}, since the coefficients of $\tau\mapsto P(\tau,i\bxi)$ are polynomials in $\bxi$ and because
 $(c_\bxi^{-1}) \in \cals'(\mZ^n)$, it follows that $(\lambda_{j,\bxi})\in \cals'(\mZ^n)$, that is, there
 exists a $k\in \mN$ such that for all $\bxi \in \mZ^n$, and all $j=1,\cdots, m_\bxi$, $|\lambda_{j,\bxi}|\leq k(1+\nm\bxi\nm)^k$.
 Thus it is enough to show that for every $(T_\bxi)\in Y$, there exists a $(S_\bxi)\in Y$ such that
 for all $\bxi \in \mZ^n$,
 $$
 \left\{\begin{array}{cl}
  S_\bxi=T_\bxi & \textrm{ if }m_\bxi=0,\phantom{\displaystyle\frac{a}{b}}\\
  \left(\displaystyle\frac{d}{dt}-\lambda_{1,\bxi}\right) S_\bxi  =T_\bxi & \textrm{ if }m_\bxi >0.
        \end{array}
\right.
$$
 Then this process can be inductively continued for the other $j$'s.

 \medskip

  \noindent (b) For $\lambda \in \mC$, let
  \begin{eqnarray*}
   E_+&=& Y(t)e^{\lambda t},\\
   E_-&=& -Y(-t)e^{\lambda t},
  \end{eqnarray*}
  where $Y(\cdot)$ denotes the Heaviside step function.
  Let $\chi_+\in C^\infty(\mR)$ be such that
  $$
  \chi_+(t)=\left\{\begin{array}{cl}
                    1 &\textrm{if } t\geq 0,\\
                    0 & \textrm{if }t\leq -1,
                   \end{array}\right.
  $$
  and define $\chi_- \in C^\infty(\mR)$ by $\chi_+ +\chi_- =1$ on $\mR$. For $U\in \calS'(\mR)$,
  define $R_\lambda(U)$ by
  $$
  R_\lambda(U)=\left\{\begin{array}{cl}
                       U\ast E_+&  \textrm{if Re}(\lambda)<0,\\
                       U\ast E_-&  \textrm{if Re}(\lambda)>0,\\
                       (\chi_+ \cdot U)\ast  E_+ + (\chi_- \cdot U)\ast  E_-&  \textrm{if Re}(\lambda)=0.
                      \end{array}\right.
$$
(Note that for $\pm \textrm{Re}(\lambda)<0$, since $E_{\pm}\in \calO_C'(\mR)$,
the space of distributions rapidly decreasing at infinity, it follows that $R_\lambda(U)$ is well-defined.) Then there holds that
$$
\left(\displaystyle\frac{d}{dt}-\lambda \right)R_\lambda(U)=U.
$$
For $U\in \calS'(\mR)$, $\textrm{Re}(\lambda)<0$ and $\varphi\in \calS(\mR)$,
$$
\langle \varphi, R_\lambda (U)\rangle
=
\langle \varphi, U\ast E_+ \rangle
=
\langle \varphi \ast \check{U} , E_+\rangle
=
\int_0^\infty (\varphi \ast \check{U})(t) e^{\lambda t} dt.
$$
Thus $\varphi \ast \check{U}\in \calO_{M}$ (smooth functions slowly increasing at infinity), and so there exists a
$k\in \mN$ such that for all $t\in \mR$,
$$
|(\varphi \ast \check{U})(t)|\leq k (1+|t|)^k.
$$
Consequently, when $\textrm{Re}(\lambda)<0$, we have
$$
|\langle \varphi, R_\lambda (U)\rangle |
\leq
k \int_0^\infty (1+t)^k |e^{\lambda t}|dt
\leq
C_{k} (1+(-\textrm{Re}(\lambda))^{-k-1}).
$$
Now the second assumption that $(d_\bxi^{-1})\in \cals'(\mZ^n)$ will allow us to
conclude that for $(T_\bxi) \in Y$,
$$
S_\bxi:= \left\{\begin{array}{cl}
  T_\bxi & \textrm{ if }m_\bxi=0,\\
  R_{\lambda_{1,\bxi}} (T_\bxi)  & \textrm{ if }m_\bxi >0
        \end{array}
\right.
$$
belongs to $Y$ as well. The details are as follows. For $(T_\bxi)\in Y$, $\varphi \in \calS(\mR)$, there exist
$k\in \mN$, $C_{\varphi,k}\in \mR$ such that for all $t\in \mR$,
 $|(\varphi \ast \check{T_\bxi}) (t)|\leq C_{\varphi,k} \cdot  (1+\nm \bxi\nm )^k (1+|t|)^k$. So for
 $\bxi$ such that $\textrm{Re}(\lambda_{1,\bxi})\neq 0$, we have
 \begin{eqnarray}
  \nonumber
|\langle \varphi ,R_{\lambda_{1,\bxi}}(T_\bxi)\rangle |
&\leq &
C_{\varphi, k} \cdot  (1+\nm \bxi\nm)^k
\cdot \int_0^\infty (1+t)^k e^{-|\textrm{Re}(\lambda_{1,\bxi})| t } dt
\\\nonumber
&\leq &
C_{\varphi, k} \cdot  (1+\nm \bxi\nm)^k
\cdot  \widetilde{C}_k \cdot (1+ |\textrm{Re}(\lambda_{1,\bxi})|^{-k-1} )\phantom{\int_a^b}
\\\label{estimate_1}
&\leq &
C_{\varphi, k} \cdot  (1+\nm \bxi\nm)^k
\cdot  \widetilde{C}_k \cdot (1+ d_\bxi^{-k-1} ).\phantom{\int_a^b}
 \end{eqnarray}
On the other hand, when $\textrm{Re}(\lambda_{1,\bxi})=0$ and $\varphi \in \calS(\mR)$,
we have  $(\varphi \ast \check{E}_+)\cdot \chi_{+}\in \calS(\mR)$
because for example using the fact that for each $k$ there exists a constant $C$ such that
for all $t,\tau\geq 0$, $|\varphi(t+\tau)|\leq C\cdot (1+t+\tau )^{-k-1}$, we obtain that
 for all $k\in \mN$, there exists $ C_k>0$ such that for all $t\geq 0, $
$$
 |(\varphi \ast \check{E}_+)(t)|
 =
 \left|\int_{-\infty}^\infty \varphi (t+\tau) E_+(\tau)d\tau\right|
=\left|\int_0^\infty \varphi (t+\tau) e^{\lambda_{1,\bxi} \tau} d\tau\right|\leq C_k \cdot (1+t)^{-k}.
$$
Similarly, $(\varphi \ast \check{E}_-)\cdot \chi_{-}\in \calS(\mR)$ as well.
Furthermore,  the $\calS$-seminorms of $\psi=(\varphi \ast \check{E}_\pm)\cdot \chi_{\pm}$ grow at most polynomially  in $|\lambda_{1,\bxi}|$. This gives
\begin{eqnarray}
 \nonumber
 |\langle \varphi, R_{\lambda_{1,\bxi}} (T_\bxi)\rangle |
 &\leq &
 |\langle (\varphi \ast \check{E}_+)\cdot \chi_{+},T_\bxi \rangle |
 +
 |\langle (\varphi \ast \check{E}_-)\cdot \chi_{-},T_\bxi \rangle |
 \\
 \label{estimate_2}
 &\leq & k\cdot (1+\nm \bxi\nm )^k \cdot (1+|\lambda_{1,\bxi}|)^k ,
\end{eqnarray}
for $\textrm{Re}(\lambda_{1,\bxi})=0$ and a suitable $k\in \mN$ which depends on $\varphi$ and $(T_\bxi)$.
 Using $(d_\bxi^{-1})\in \cals'(\mZ^n)$ and the estimates \eqref{estimate_1} and \eqref{estimate_2}, this completes the proof.
\end{proof}

Let us finally go back to the original problem.  We consider temperate distributions which are periodic in the
space variables, i.e., distributions in
$$
 \widehat Y=\{T\in \calS'(\mR^{n+1}_{t,x}): \forall a\in A\mZ^n: T(t,x+a)=T(t,x)\}
$$
where $A\in\mR^{n\times n}$ is a non-singular matrix. Upon using partial Fourier transform with respect to
$x,$ we obtain the isomorphism
$$
F:\widehat Y\overset\sim\longrightarrow Y:T\longmapsto (T_\bxi)_{\bxi\in\mZ^n}, \text{ where }(\calF_xT)(\eta)=
  \sum_{\bxi\in\mZ^n} T_\bxi(t)\delta(\eta-B\bxi)
$$
with $B=2\pi (A^{-1})^\top$ (where $\cdot^\top$ is used to denote the transpose)  and
$$
\langle\phi,\calF_x T\rangle=\langle\int_{\mR^n}\phi(t,\eta) e^{-i x\eta}\,d\eta, T\rangle\text{ for }\phi\in\calS(\mR^{n+1}).
$$
\par
Since $F(P(\partial)T)=\bigl(P(\frac d{dt}, iB\bxi) T_\bxi\bigr)_{\bxi\in\mZ^n}$ holds if
$FT=(T_\bxi)_{\bxi\in\mZ^n},$ we obtain the following corollary to Theorem 1.2.
\begin{corollary}
Let $P(\partial)\in\mC[\partial_t,\partial_1,\dots,\partial_n],$ $A\in\text{\rm Gl}_n(\mR),$ $B=2\pi (A^{-1})^\top$ 
and $\widehat Y$ as above.
We assume, furthermore, that
$$
P(\tau,iB\bxi)=c_\bxi\cdot \prod_{j=1}^{m_\bxi} (\tau-\lambda_{j,\bxi}),
 $$
with $m_\bxi \in \mN_0$, $c_\bxi \in \mC\setminus \{0\}$, $\lambda_{1,\bxi},\cdots, \lambda_{m_\bxi, \bxi}\in \mC$,
and that the conditions {\rm (1)} and {\rm (2)} in Theorem 2.1 hold.
  Then $P(\partial):\widehat Y\longrightarrow\widehat Y$ is surjective.
\end{corollary}


\begin{thebibliography}{99}

\bibitem{Ehr}
Leon Ehrenpreis.
Solution of some problems of division. I. Division by a polynomial of derivation.
{\em American Journal of Mathematics},  76:883-903, 1954.

\bibitem{Ehrb}
Leon Ehrenpreis.
Solution of some problems of division. II. Division by a punctual distribution.
{\em American Journal of Mathematics}, 77:286-292, 1955.

\bibitem{Ehrc}
Leon Ehrenpreis.
Solutions of some problems of division. III. Division in the spaces, ${\mathscr D}',{\mathscr H}, {\mathscr Q}_A, {\mathscr O}$.
{\em American Journal of Mathematics}, 78:685-715, 1956.

\bibitem{Ehrd}
Leon Ehrenpreis.
Solution of some problems of division. IV. Invertible and elliptic operators.
{\em American Journal of Mathematics}, 82:522-588, 1960.

\bibitem{Ehre}
Leon Ehrenpreis.
Solutions of some problems of division. V. Hyperbolic operators.
{\em American Journal of Mathematics}, 84:324-348, 1962.

\bibitem{Hor}
Lars H{\"o}rmander.
On the division of distributions by polynomials.
{\em Arkiv f\"or Matematik}, 3:555-568, 1958.

\bibitem{Loj}
Stanis{\l}aw {\L}ojasiewicz.
Sur le probl\`eme de la division.
{\em Studia Mathematica}, 18:87-136, 1959.

\bibitem{Sas}
Amol Sasane.
Algebraic characterization of autonomy and controllability
of behaviours of spatially invariant systems.
 {\tt http://arxiv.org/abs/1208.6496}

\bibitem{Tre}
Fran{\c{c}}ois Tr{\`e}ves.
{\em Topological vector spaces, distributions and kernels.}
Unabridged republication of the 1967 original.
Dover Publications, Mineola, NY, 2006.

\bibitem{Wag}
Peter Wagner.
A new constructive proof of the Malgrange-Ehrenpreis theorem.
{\em The American Mathematical Monthly}, 116:457-462, 2009.

\bibitem{Woo}
Jeffrey Wood.
Key problems in the extension of module-behaviour duality.
Fourth special issue on linear systems and control.
{\em Linear Algebra and its Applications}, 351/352:761-798, 2002.

\end{thebibliography}
\end{document}